\documentclass[11pt, reqno]{amsart}

\usepackage[utf8]{inputenc}

\title{Query complexity and the polynomial Freiman--Ruzsa conjecture}
\author{Dmitrii Zhelezov}
\author{D\"{o}m\"{o}t\"or P\'alv\"olgyi}

\date{}

\usepackage{amsmath}
\usepackage{amssymb}
\usepackage{graphicx}
\usepackage[colorinlistoftodos]{todonotes}
\usepackage[colorlinks=true, allcolors=blue]{hyperref}

\newtheorem{lemma}{Lemma}[section]

\newtheorem{Theorem}{Theorem}[section]
\newtheorem{Proposition}[Theorem]{Proposition}

\newtheorem{Definition}[Theorem]{Definition}
\newtheorem{Conjecture}{Conjecture}

\theoremstyle{remark}
\newtheorem{Claim}[Theorem]{Claim}

\newtheorem{Remark}{Remark}[section]

\newcommand{\mc}[1]{\mathcal{#1}}

\newcommand{\cj}[1]{%
  \overline{#1}%
}


\begin{document}

\maketitle

\begin{abstract}
    We prove a query complexity variant of the weak polynomial Freiman--Ruzsa conjecture in the following form. For any $\epsilon > 0$, a set $A \subset \mathbb{Z}^d$ with doubling $K$ has a subset of size at least $K^{-\frac{4}{\epsilon}}|A|$ with coordinate query complexity at most $\epsilon \log_2 |A|$.
    
    We apply this structural result to give a simple proof of the ``few products, many sums'' phenomenon for integer sets. The resulting bounds are explicit and improve on the seminal result of Bourgain and Chang. 
\end{abstract}

\section*{Notation and preliminaries}
The following notation is used throughout the paper. The expressions $X \gg Y$, $Y \ll X$, $Y = O(X)$, $X = \Omega(Y)$ all have the same meaning that there is an absolute constant $c$ such that $Y \leq cX$. Further, $X \gg_\epsilon Y$ means that there is a function $c(\cdot)$ such that
$$
Y \leq c(\epsilon) X,
$$
and the same convention applies for the $\gg, O, \Omega$-notation.

$X \geq Y^{c-o(1)}$ means that $X \gg_\epsilon Y^{c-\epsilon}$ for any $\epsilon > 0$.

If $X$ is a set then $|X|$ denotes its cardinality. 

Let $G$ be an additive torsion-free group and $A, B \subset G$. For concreteness, we will assume henceforth that $G = \mathbb{Z}^{d}$ for some unspecified dimension $d$. We will also assume that $G$ is embedded into an ambient vector space $\mathbb{Q}^d$, making no distinction between $G$ and the embedding. In particular, we fix a standard basis $\{\vec{e}_1, \ldots, \vec{e}_d \}$ and define coordinate projections  $\pi_i: G \to \mathbb{Z}$ by
$$
\pi_i (n_1 \vec{e}_1 + \ldots + n_d \vec{e}_d ) = n_i.
$$

The \emph{sumset} $A + B$ is defined as the set of all pairwise sums  
$$
A+B := \{ a + b: a \in A, b \in B \}.
$$
The set of pairwise products (or the \emph{product set}) $AB$ is defined \emph{mutatis mutandis} with.

The $\lambda_k$ constant of an integer set $A$ is defined as
\begin{equation} \label{eq:lambda_constant}
\lambda_k(A) := \max \left\|\sum_{n \in A} c_n e^{2 \pi i nx}  \right\|^2_{L_{2k}([0,1])},
\end{equation}
where $\max$ is taken over positive weights $\{c_n\}_{n \in A}$ with $\sum_n c_n^2 = 1$. It is related to a more commonly used notion of the additive energy or order $k$, defined as 
$$
E_k(A) := \left\|\sum_{n \in A} e^{2 \pi i nx}  \right\|^{2k}_{L_{2k}([0,1])}.
$$
Note, the the sum above simply counts the number of $2k$-tuples $(a_1, \ldots, a_{2k}) \in A^{2k}$ such that $a_1 + \ldots + a_k = a_{k+1} + \ldots + a_{2k}$. 
By the Cauchy-Schwartz inequality, one immediately arrives at
$$
|\underbrace{A+\ldots+A}_\text{$k$ times}| \geq \frac{|A|^{2k}}{E_k(A)}.
$$
More information on additive energies can be found in the book \cite{TaoVu}.

Taking $c_n = 1/\sqrt{|A|}$ in the definition of $\lambda_k(A)$, one arrives at the estimate 
$$
E^{1/k}_k(A) \leq |A|\lambda_k(A),
$$
and indeed one can bound the additive energy of any $A' \subset A$ in a similar way. 

It will be convenient to use the standard shortcut $e(x) := e^{2 \pi i x}$. 

\section{Introduction}

\subsection{Weak polynomial Freiman--Ruzsa conjecture}

One of the main research avenues of additive combinatorics is to extract structural information about sets with \emph{small doubling} $K$ defined as 
$$
K := \frac{|A+A|}{|A|}.
$$
A fundamental result, known as Freiman's lemma \cite{TaoVu}, asserts that $A$ is always contained in an affine subspace of dimension at most $O(K)$.

A very rigid structure can be deduced when $K = o(\log |A|)$ using the much harder quantitative Freiman Theorem, see \cite{Sanders_2012} for the state-of-the art bounds and background. 

However, very little is known in the regime $\log K \gg \log |A|$ and one of the central problems in the area is to close this gap (i.e. to obtain polynomial bounds). The Polynomial Freiman--Ruzsa conjecture predicts, informally, that there is a subset $A' \subset A$ of size at least $K^{-O(1)}|A|$, such that $A'$ (after a suitable transformation) is contained in a convex body of dimension $O(\log K)$ and volume $K^{O(1)} |A|$. 

It turns out that for many applications (see \cite{Chang_2009}) it would suffice that the weaker form below holds true.

\begin{Conjecture}[Weak Polynomial Freiman--Ruzsa Conjecture] \label{conj:WPFR}
  For any set $A$ with doubling $K$ there is a subset $A'$ of size $K^{-O(1)}|A|$ contained in an affine subspace of dimension $O(\log K)$.
\end{Conjecture}

We make a step towards Conjecture~\ref{conj:WPFR} by replacing the rank condition with a weaker property of logarithmic \emph{coordinate query complexity}. The definition we use is different from the one commonly used in computer science as we assume that a single query outputs an integer number rather than a $\{0, 1\}$ bit. It is defined as follows.

Assume Alice and Bob agree on some large set $X \subset \mathbb{Z}^d$. Next, Alice chooses an element $x \in X$ and keeps it in secret. Bob tries to guess $x$ by probing the value of $\pi_i(x)$ for some $i$, one coordinate at a time. The coordinate query complexity of $X$ is then the maximal number of coordinate queries Bob should perform in order to recover $x$ in the worst case. 

\begin{Theorem}[Query-complexity PFR] \label{thm:query_PFR}
   For any $\epsilon > 0$ the following holds. For any set $A \subset \mathbb{Z}^d$ with $|A+A| \leq K|A|$ there is a subset of size at least $K^{-\frac{2}{\epsilon}}|A|$ with coordinate query complexity at most $\epsilon \log_2 |A|$. 
\end{Theorem}

Note that Conjecture~\ref{conj:WPFR} would immediately imply Theorem~\ref{thm:query_PFR}. Indeed, assume $A' \subset A$ is contained in an affine subspace $V$ of dimension $s = O(\log K)$. Then it follows from basic linear algebra that there are $s$ coordinates $i_1, \ldots, i_s$ such that the map $v \mapsto (\pi_{i_1}(v), \ldots, \pi_{i_s}(v))$ is injective on $V$. Thus, Bob can recover any $a \in A'$ by probing at most $s$ coordinates. 

At the same time Theorem~\ref{thm:query_PFR} is, to our knowledge, the first result sensitive enough to detect a large structured piece inside a set $A$ with $\log K \gg \log |A|$.

\subsection{Few products, many sums}

We apply Theorem~\ref{thm:query_PFR} in the second part of the paper in order to give improved bounds for the ``few products, many sums'' phenomenon for integer sets, sometimes called the weak Erd\H{o}s--Szemer\'edi conjecture. The state-of-the art bounds, due to Bourgain and Chang \cite{bourgain2004size}, were obtained using a \emph{tour de force} induction on scales argument and are rather inefficient.

The sum-product problem is concerned with showing that either the set of sums or the set of products is always large. It was conjectured by Erd\H{o}s and Szemer\'{e}di \cite{ES} that, for all $\epsilon>0$ and any finite $A \subset \mathbb Z$,
\begin{equation}
\max \{|A+A|,|AA|\} \geq c(\epsilon)|A|^{2-\epsilon}
\label{ESconj}
\end{equation}
where $c(\epsilon)>0$ is an absolute constant. The same conjecture can also be made over the reals, and indeed other fields. The  Erd\H{o}s--Szemer\'{e}di conjecture remains open, and it appears to be a deep problem. Konyagin and Shkredov \cite{KS} proved that \eqref{ESconj} holds with $\epsilon <2/3$, and the current best bound, due to Rudnev and Stevens \cite{RudnevStevens}, has $\epsilon \leq 2/3 - 2/1167 +o(1)$. These bounds hold over real numbers, and their proofs are geometric in nature. 

It turns out that geometric arguments are only efficient when $|A+A|$ is small. Elekes and Ruzsa \cite{Elekes_Ruzsa_2003} proved that for any set $A$ of real numbers 
$$
|AA||A+A|^4 \gg |A|^{6-o(1)},
$$
thus confirming the Erd\H{o}s--Szemer\'edi conjecture in the regime $|A+A| \ll |A|^{1+o(1)}$. This particular case is known as the ``few sums, many products'' phenomenon.

Surprisingly enough, the dual ``few products, many sums'' case of the Erd\H{o}s--Szemer\'edi conjecture remains open for sets of real numbers and is sometimes dubbed as the weak Erd\H{o}s--Szemer\'edi conjecture. The best bound for real numbers is due to Murphy et al.~\cite{MRSS_few_products_many_sums}, who proved that if $|AA| \ll |A|^{1+o(1)}$ then
$$
|A+A| \gg |A|^{8/5-o(1)}.
$$

The weak Erd\H{o}s--Szemer\'edi conjecture has been resolved by Bourgain and Chang \cite{bourgain2004size} for integer sets. They proved that, for any $\epsilon > 0$ there is $C(\epsilon)$ such that for any integer set $A$
\begin{equation} \label{eq:B_C_sumprod}
|A+A| \gg K^{C(\epsilon)}|A|^{2-\epsilon},
\end{equation}
with $K = |AA|/|A|$. In other words, writing $K = |A|^\delta$,  
$$
|A+A| \gg |A|^{2-\epsilon(\delta)}
$$
with $\epsilon(\delta) \to 0$ as $\delta \to 0$.

It was also proved in \cite{bourgain2004size} that 
\begin{equation} \label{eq:B_C_kfoldsumprod}
\max \{|\underbrace{A+\ldots+A}_\text{$k$ times}|, |\underbrace{A\ldots A}_\text{$k$ times}| \} \geq |A|^{b(k)}
\end{equation}
with $b(k) \gg \log^{1/4} k$.

The dependence $C(\epsilon)$ in (\ref{eq:B_C_sumprod}) is rather poor since the argument in \cite{bourgain2004size} relies on an intricate induction on scales device. Theorem~\ref{thm:query_PFR} applied to the prime valuation image of $A$ allows one to bypass such complications since it is agnostic with regards to the dimension of the ambient space.

\begin{Theorem}[Few products, many sums] \label{thm:few_prods_many_sums}
  For any $1 > \epsilon > 0$ the following holds. Let $A \subset \mathbb{Z}$ and 
  $$
    K_* := \frac {|AA|}{|A|}.
  $$
  Then 
  $$
  |kA| \gg_{k} |A|^{k -  2 \epsilon k \log_2 k}K^{-2k/\epsilon}_*.
  $$
\end{Theorem}
\begin{Remark}
   Theorem~\ref{thm:few_prods_many_sums} can be extended to sets $A$ of algebraic numbers of degree $O(\log |A|)$, applying almost verbatim the ideas of \cite{bourgain2009sum}. However, in order to resolve the weak Erd\H{o}s--Szemer\'edi conjecture for sets of real (or complex) numbers one has to rule out the case when $A$ consists of units in a number field of very large degree (cf. Proposition 10 of \cite{bourgain2009sum}). For example, a resolution of Conjecture 2.9 (``Log-span conjecture'') of \cite{RSZ} would provide such a tool.
\end{Remark}

To formulate Theorem~\ref{thm:few_prods_many_sums} in its most general form we need to introduce the notion of additive (or mutliplicative) tripling introduced in \cite{RSZ}.

Let $(G, +)$ be an abelian group. For a set $U \subset G$ define 
$$
\beta(U) := \inf_{A_1, A_2} \frac{|A_1 + A_ 2 + U|}{|A_1|^{1/2}|A_2|^{1/2}}.
$$

The following Lemma, proved in Statement 3.3 of \cite{RSZ}, connects $\beta$ with the usual notion of additive doubling.
\begin{lemma}[$\beta$ bounds additive doubling, \cite{RSZ}] \label{lm:beta-pluneccke}
Let $A$ be an additive set and write 
$K_+ := |A+A|/|A|$. Then for any $U \subset A$ holds
\begin{equation} \label{eq:beta_plunecke}
\beta(U) \leq K_+^2.
\end{equation}
\end{lemma}

We invite the reader to consult \cite{RSZ} for a detailed treatment of $\beta$ and related quantities under the umbrella term ``induced doubling''. 

In what follows we are going to use $\beta$ mostly with respect to multiplication, and so to avoid confusion will write $\beta_*$ in such cases. 
 
\begin{Theorem}[Few products, many sums for $\beta_*$ and $\lambda_k$] \label{thm:few_prods_many_sums_lambda}
  For any $1 > \epsilon > 0$ the following holds. Let $A \subset \mathbb{Z}$.
  Then 
  $$
  \lambda_k(A) \leq 10 \beta^{\frac{1}{\epsilon}}_* |A|^{2\epsilon \log_2 k},
  $$
  where 
  \begin{equation} \label{eq:mult_beta}
  \beta_*(A) := \inf_{B, C \subset \mathbb{Z}} \frac {|ABC|}{|B|^{1/2}|C|^{1/2}}.
  \end{equation} 
\end{Theorem}

A corollary of Theorem~\ref{thm:few_prods_many_sums_lambda} is the following $k$-fold sum-product estimate, which improves on the state of the art bound in \cite{bourgain2004size}. 
\begin{Theorem}[Iterated sum-product] \label{thm:iterated_sum_product}
   For all $k \in \mathbb{N}, k > 2$ and $A \subset \mathbb{Z}$ the following holds. Let 
   $$
   \delta := \frac {\log \beta_*(A)}{\log |A|}
   $$
   with $\beta_*$ defined by (\ref{eq:mult_beta}).
   Then assuming $|A|$ is large enough
   $$
   |kA| = |\underbrace{A+\ldots+A}_\text{$k$ times}| \geq |A|^{k - 10k\sqrt{\delta \log_2 k}}
   $$
   and
   $$
   |A^{(k)}| = |\underbrace{A \ldots A}_\text{$k$ times}| \geq |A|^{\delta \log_2 k}.
   $$

In particular, there is an absolute $c > 0$  ($c = 10^{-4}$ would do), such that either 
$$
|A^{(k)}| \geq |A|^{b(k)} 
$$
or
$$
|kA| \geq |A|^{b(k)},
$$
with $b(k) = \frac{c \log_2 k}{\log_2 \log_2 k}$.

\end{Theorem}

This improves on \cite{bourgain2004size} where a similar bound with $b(k)$ of order $\log^{1/4} k$ was obtained. 

The growth of $b(k)$ is essentially the best one can hope for, as shown by the following example. The authors are indebted to an anonynous referee who brought this example to our attention.

\begin{Proposition}
There is an absolute constant $C$ with the following property. Let $k \in \mathbb{N}$. Then there is a set $A \subset \mathbb{Z}$ such that $|kA| + |A^{(k)}| \leq |A|^{C \log k/\log \log k}$.
\end{Proposition}
\begin{proof}
It is essentially due to Erd\H{o}s and Szemer\'edi \cite{ES}. 

In what follows, $c_1,c_2,\dots$ are absolute constants which we do not bother to specify explicitly. Assume $k$ is large and set
\[ A := \{ \prod_{i} p_i^{e_i} :  e_i \leq (\log k)^{1/2}\},\] where the product is over the primes $p_i$ of size at most $(\log k)^{1/2}$.  We have
\[  e^{c _1(\log k)^{1/2}} \leq |A| \leq (\log k)^{\frac{1}{2} \pi (\log^{1/2} k)}.\]
Also,
\[ \max A \leq (\prod_i p_i)^{(\log k)^{1/2}} \leq e^{2 \log^{1/2} k \pi(\log^{1/2} k) \log \log k} \leq k^{c_2},\] and therefore 
\[ |kA| \leq k^{c_2 + 1} \leq |A|^{c_3 (\log k)^{1/2}}.\]
On the other hand 
\[ A^{(k)} := \{ \prod_{i} p_i^{e_i} :  e_i \leq k(\log k)^{1/2}\},\] 
and so
\[ |A^{(k)}| \leq \big( k (\log k)\big)^{\frac{1}{2} \pi ((\log k)^{1/2})} \leq e^{c_4 \frac{(\log k)^{3/2}}{\log \log k}} \leq |A|^{c_5 \frac{\log k}{\log \log k}}. \]
\end{proof}

\section{Proof of Theorem~\ref{thm:iterated_sum_product}}

Before moving forward, let's deduce Theorem~\ref{thm:iterated_sum_product} assuming  Theorem~\ref{thm:few_prods_many_sums_lambda} holds true.

Write $\beta_*(A) = |A|^\delta$. First, let's record the following one liner:

\begin{Claim}[Iterated $\beta_*$]  \label{claim:iterated_beta}
For any integer $t > 1$
$$
|A^{(2^t -1)}| \geq \beta_*^t
$$
\end{Claim}
\begin{proof}
For $t=2$ the claim follows from the definition of $\beta_*$. For $t > 2$ one has by induction
$$
|A^{(2^t -1)}| = |A^{2^{(t-1)} - 1}A^{2^{(t-1)} -1}A| \geq \beta_*|A^{2^{(t-1)} -1}| \geq \beta_*^t.
$$
\end{proof}

At the expense of decreasing the constants in $b(k)$, from now on let's assume that $k = 2^t$ for some integer $t$. Henceforth $c > 0$ is assumed to be a small fixed constant to be defined in due course.

Then if $\delta \geq c / \log_2 t$, by Claim~\ref{claim:iterated_beta},
$$
|A^{(k)}| = |A|^{2^t} \geq \beta_*(A)^t \geq |A|^{c t/\log_2 t} \geq |A|^{b(k)}.
$$

Otherwise, from the definition of $\lambda_l$ with all the weights equal to $|A|^{-1/2}$ and the Cauchy-Schwartz inequality,
$$
|lA| \geq \frac{|A|^{l}}{\lambda_l(A)^l}.
$$
Applying Theorem~\ref{thm:few_prods_many_sums_lambda} with $\beta_*(A) = |A|^\delta$ one has for any $\epsilon > 0$ and $l > 0$
$$
|lA| \geq 10^{-l}|A|^{l}|A|^{-\delta l/\epsilon - 2l \epsilon \log_2 l}.
$$
Taking $\epsilon = (\delta \log_2 l)^{-1/2}$ and assuming $|A|$ is large enough, we get
$$
|lA| \geq |A|^{l - 10l\sqrt{\delta \log_2 l}} \geq |A|^{l - 10l\sqrt{c \log_2 l/\log_2 t}}.
$$
Recalling that $k = 2^t$ we further estimate very crudely with $l=t$
$$
|kA| \geq |tA| \geq |A|^{t - 10tc^{1/2}} \geq |A|^{t/2} \geq |A|^{b(k)}.
$$

\section{Background on $\beta$ and quasi-cubes}

A \emph{quasicube} is a generalization of the binary cube $\{0, 1\}^d$ and is defined recursively as follows. 

\begin{Definition}[Quasicubes]
We say that a set $H \subset \mathbb{Z}^d$ is a \emph{quasicube} if there is a coordinate projection $\pi_i$ such that $|\pi_i(H)| = 2$ and  either 
\begin{enumerate}
    \item $\pi_i$ is injective
    \item $\pi_i(H) = \{x, y \}$ and both 
    $\pi^{-1}(x)$ and $\pi^{-1}(y)$ are quasicubes.
\end{enumerate}
\end{Definition}

The following theorem was proved in \cite{RSZ}, and a short self-contained proof can be found in \cite{GMRSZ}.

\begin{Theorem}[Subsets of quasicubes have large $\beta$, \cite{RSZ}] \label{thm:beta}
  Let $H$ be an arbitrary quasicube. Then for any $U \subset H$ holds
  \begin{equation} \label{eq:beta_estimate_hypercube}
  \beta(U) = |U|.
  \end{equation}
\end{Theorem}
The power of Theorem~\ref{thm:beta} is that the estimate (\ref{eq:beta_estimate_hypercube}) depends neither on the dimension of the ambient space nor on the density of $U$ in $H$. 



\section{Branching depth and binary subtrees}

Let $T$ be a rooted tree. We will write $L(T)$ for the set of leaves. We further define the following quantities.

\begin{Definition}[Branching depth]
Let $r$ be the root of $T$ and write $P(l \to r)$ for the set of vertices on the (unique) path from $l$ to $r$. Let $d_l$ be the number of vertices in $P(l \to r)$ with at least two children.
Then the \emph{branch-depth} of $T$ is defined as 
$$
d(T) := \max_{l \in L(T)} d_l.
$$
\end{Definition}

\begin{Definition}[Largest binary subtree]
    Let us call a rooted tree \emph{binary} if each node has at most two children.
    Further, for a tree $T$ write 
    $$
    b(T) = \max_{ \text{ binary } T' \subset T} |L(T')|.
    $$
\end{Definition}

The strategy for the rest of the argument is to prove that for any tree $T$ either there is a large subtree $T'$ with $d(T') = o(\log |L(T)|)$ or $\log b(T') \gg \log |L(T)|$. 

Let $T$ be a tree with $N := |L(T)|$. Fix $\epsilon > 0$. 

\begin{Definition}[Largest $\epsilon$-low subtree]
    Let us call a rooted tree $T'$ \emph{$\epsilon$-low} if $d(T') \leq \epsilon \log_2 N$.
    Further, for a tree $T$ write 
    $$
    D_\epsilon(T) = \max_{ \text{ $\epsilon$-low } T' \subset T} |L(T')|.
    $$
\end{Definition}


\begin{lemma}[Low vs binary subtree alternative] \label{lm:tree_bound}
   For any tree $T$ and $1\ge \epsilon > 0$
   $$
       D_\epsilon(T) b^{1/\epsilon} (T) \geq |L(T)|.
   $$
\end{lemma}
\begin{proof}
   The proof is by induction on the height of $T$. Write $N := |L(T)|$. For a single root or a root with a single child the inequality is trivial. For a tree of height $1$ with at least two children we have $b(T) = 2$. If $\epsilon \log_2 N \geq  1$, then the whole tree is $\epsilon$-low and we are done. Otherwise, 
   $$
   b^{1/\epsilon}(T) \geq N.
   $$
   
   Now assume the height is larger. Without loss of generality we may assume that the root has at least two children. Let $T_i$ be the subtrees rooted at the children, write $D_i := D_\epsilon(T_i)$, $N_i := |L(T_i)|$, $b_i := b(T_i)$ and $b := b(T)$.
   
   Call $T_i$ \emph{small} if $N_i \leq 2^{-1/\epsilon} N$, otherwise call it \emph{big}. Denote the families of these subtrees by $S$ and $B$, respectively.
   
   

   \textbf{Claim.}
   If there are no big subtrees, i.e., $\sum_{T_i \in B} N_i= 0$, we are done. Indeed, in this case the branching depth of the tree constructed by attaching the maximal $\epsilon$-low trees in $T_i$ to the root of $T$ is at most 
   \begin{equation} \label{eq:low_tree_depth}
        \epsilon \max_i \log_2 N_i + 1 \leq \epsilon \log_2 N,
   \end{equation}
   so 
   $$
   D_\epsilon(T) \geq \sum_i D_\epsilon (T_i) \geq \frac{\sum_i N_i}{\max_i b^{1/\epsilon}_i} = \frac{N}{\max_i b^{1/\epsilon}_i}.
   $$
   But clearly $b^{1/\epsilon}(T) \geq \max_i b^{1/\epsilon}_i$ and the induction is closed. \\
  
   \textbf{Claim.} 
   If there are at least two big subtrees, we are also done. Indeed, let $T_i, T_j \in B$. Without loss of generality, $b_j \geq b_i$, so
   $$
   b(T) \geq b_{i} + b_j \geq 2b_{i}.
   $$
   Clearly, $D_\epsilon(T) \geq D_\epsilon(T_{i})$ and thus
   $$
   D_\epsilon(T)b^{1/\epsilon}(T) \geq D_\epsilon(T_{i})2^{1/\epsilon} b^{1/\epsilon}_{i}  \geq 2^{1/\epsilon}N_{i} \geq N.
   $$
   \\\noindent
   So the remaining case is when $T_1$ is a single big subtree, so $N_1 > 2^{-1/\epsilon}N$. Write 
   $$
   b_M = \max \{b_i : T_i \text{ is small } \}.
   $$
   Let $c > 0$ be such that $b_M = c b_1$. We have $b(T) \geq (1+c)b_1=(1+\frac 1c)b_M$. Since $D_\epsilon(T) \geq D_1$, we are done unless 
   $$
   ((1+c)b_1)^{1/\epsilon} D_1 \leq N.
   $$
   
   By induction, the left hand side is at least $(1+c)^{1/\epsilon}N_1$, so it must be 
   \begin{equation} \label{eq:low_bound_N}
   (1 + c)^{1/\epsilon} \leq \frac{N}{N_1}  < 2^{1/\epsilon}.
   \end{equation}
   
   In particular, $c < 1$ and $b_M < b_1$.
   
   Since attaching to the root of $T$ increases the branch-depth  by at most one, it follows similarly to (\ref{eq:low_tree_depth}) that
   $$
   D_\epsilon(T) \geq \sum_{T_i \in S} D_\epsilon(T_i) \geq \frac{\sum_{T_i \in S} N_i}{b^{1/\epsilon}_M}.
   $$
   Thus,
   $$
   D_\epsilon(T) b^{1/\epsilon}(T) \geq \frac{\sum_{T_i \in S} N_i}{b^{1/\epsilon}_M}b^{1/\epsilon}(T) \geq (1+1/c)^{1/\epsilon}(N-N_1).
   $$
   We are done if the right hand side is at least $N$, that is if
   $$
   (1+1/c)^{1/\epsilon}(N - N_1) \geq N.
   $$
   Dividing by $N_1$, this is equivalent to
    $$
   (1+1/c)^{1/\epsilon}(\frac{N}{N_1}-1) \geq \frac{N}{N_1}.
   $$
   Rearranging for $\frac{N}{N_1}$ gives that we need
     $$
     \frac{N}{N_1}\ge 
   \frac{(1+1/c)^{1/\epsilon}}{(1+1/c)^{1/\epsilon}-1}.
   $$
   Using the left inequality from (\ref{eq:low_bound_N}), it is sufficient to prove 
   $$
   (1+c)^{1/\epsilon}\ge \frac{(1+1/c)^{1/\epsilon}}{(1+1/c)^{1/\epsilon}-1}=1+\frac{1}{(1+1/c)^{1/\epsilon}-1}.
   $$
   The left hand side is decreasing in $\epsilon$ and the right hand side is increasing in $\epsilon$.
   Since they are equal for $\epsilon=1$, the inequality holds for all $\epsilon\le 1$.
\end{proof}

\section{Proof of Theorem~\ref{thm:query_PFR}} \label{sec:proof_query_PFR}

The initial step to prove Theorem~\ref{thm:query_PFR} is to transform the set $A$ in question into a rooted tree $T(A)$. We build $T(A)$ recursively. 

Let $1 \leq i \leq d$ be some coordinate index and $\pi_i$ be the corresponding coordinate projection. Any set $X$ \emph{fibers} with respect to $\pi_i$ in the sense that 
$$
X = \bigsqcup_{y \in \pi_i(X)} \pi_i^{-1}(y).
$$
We call the disjoint sets $X_y := \pi_i^{-1}(y)$ \emph{fibers} of $X$ above $y$.

Now let's get back to the construction of $T(A)$. It has a root $v_0$, 
and it is the only node of $T(A)$ is $A$ is a singleton. If not, let  $j$ be the minimal coordinate index such that $|\pi_j(A)| > 1$. We recursively attach to $v_0$ the trees $T(A_x)$ for each fiber $A_x, x \in \pi_j(A)$ induced the projection $\pi_j$. The root of $T(A_x)$ is labelled with the pair $(j, \pi_j(A_x))$. The process will terminate since the coordinate index always increases.

Since the process terminates when the fiber becomes a singleton set, the elements of $A$ are in one-to-one correspondence with the leafs of $T(A)$ endowed with the labels.

Now everything is set up for the proof of Theorem~\ref{thm:query_PFR}.
\begin{proof}
    Let $\epsilon > 0$ be fixed and $A \subset G$ be a set with 
    $K := |A+A|/|A|$. 
    
    Let $T_A$ be the rooted tree corresponding to $A$ and $T_B$ be a (one of possibly many) largest binary subtree of $T_A$. 
    
    \textbf{Claim.}
    We claim that the set $B \subset A$ which corresponds to the leaves $L(T_B)$ as described above is contained in a quasicube. The claim follows from a simple induction on the height of the tree $T_B$. Indeed, a single root or a binary tree of height one is clearly a quasicube subset. Otherwise, the root has either one or two children, and in both cases the claim follows from the definition of a quasicube.
    
    By the hypothethis of the theorem and (\ref{eq:beta_plunecke}),
    $$
    K \geq \beta^{1/2}(B) = |B|^{1/2}. 
    $$
    
    It therefore follows that 
    $$
    b(T_A) = |L(T_B)| = |B| \leq K^2.
    $$
    
    We immediately conclude by Lemma~\ref{lm:tree_bound} that 
    \begin{equation} \label{eq:tree_depth_bound}
    D_\epsilon(T_A) \geq b^{-1/\epsilon}(T_A) |L(T_A)| \geq K^{-\frac{2}{\epsilon}} |A|.
    \end{equation}
    Thus, by definition, there is a subtree $T' \subset T_A$ with branching depth at most  $\epsilon \log |A|$ and size at least $K^{-\frac{2}{\epsilon}} |A|$.  
    
    Let $A' \subset A$ be the subset corresponding to the leaves $L(T')$. In order to conclude the proof it remains to note the coordinate query complexity of $A'$ is at most the depth of $T'$. Let $x \in A'$. For any $j$ the $j$-coordinate query returns the value $\pi_j(x)$, which uniquely identifies the $\pi_j$-fiber of $x$. Thus, we can traverse $T(A')$ from the root to the unique leaf corresponding to $x$ each time taking the branch corresponding to the coordinate query. The number of queries is going to be at most the depth of $T(A')$, and we are done by (\ref{eq:tree_depth_bound}).
\end{proof}

\section{Proof of Theorem~\ref{thm:few_prods_many_sums} and Theorem~~\ref{thm:few_prods_many_sums_lambda}}

\subsection{Prime valuation mapping}

Let $A$ be a set of integers, the goal 
is essentially to prove that either $\beta_*(A)$ (that is, $\beta(A)$ with respect to multiplication) is large or $E_+(A)$ is small.

The first step is to transform $A$ into a multidimensional set using the prime valuation map which is as follows. Let $\{p_1, \ldots, p_D \}$ be the set of prime divisors of the elements in $A$. We consider the valuation map $\Pi: \mathbb{Z} \to \mathbb{Z}^D$:
$$
\Pi(a) = (v_{p_1}(a), \ldots, v_{p_D}(a))
$$
where $v_{p_i}(a)$ is the maximal power $\alpha$ such that $p^\alpha_i$ divides $a$.

Clearly for integer sets $\Pi(X) + \Pi(Y) = \Pi(XY)$ so 
$$
\beta_+(\Pi(A)) = \beta_*(A).
$$

Since $\Pi$ is one-to-one from now on we identify any  $A$ with $\mc{A} := \Pi(A) \subset \mathbb{Z}^D$. The convention is that calligraphic letters live in $\mathbb{Z}^D$ and capital italic live in $\mathbb{Z}$. We also follow that convention that $\pi_i$ is the one-dimensional projection in $\mathbb{Z}^D$ to the coordinate corresponding to the prime $p_i$.

\subsection{Chang's argument}

Let's recall Proposition~6 of \cite{chang2003erdHos}. 
\begin{Proposition} \label{prop:chang}
Let $p$ be a fixed prime, and let 
$$
F_j(x) \in \left \langle \left\{ e^{2\pi i p^j n x} | (n, p) = 1 \right\} \right \rangle^{+}. 
$$
Then 
$$
\left (\| \sum_j F_j\|_{2k} \right)^2 \leq \binom {2k}{2} \sum_j \|F_j(x) \|^2_{2k}.
$$
\end{Proposition} 
\begin{proof}
Expanding the brackets, one can write 
$$
\| \sum_j F_j\|^{2k}_{2k} = \int^1_0 |\sum_j F_j(x)|^{2k} dx
$$ 
as a sum of terms of the form 
\begin{equation} \label{eq:expansion}
    \int^1_0 F_{j_1}(x)\ldots F_{j_k}(x)\cj{F_{j_{k+1}}}(x)\ldots \cj{F_{j_{2k}}}(x) dx.
\end{equation}

We claim that if $\{ j_i \}^{2k}_{i=1}$ are all distinct, the integral above is zero. Indeed, if this is the case, the integral above can be further broken down into a sum of integrals
$$
C_{j_1,\ldots,j_{2k}}\int^1_0  e(p^{j_1} n_{j_1} + \ldots + p^{j_k} n_{j_k} - p^{j_{k+1}} n_{j_{k+1}} - \ldots - p^{j_{2k}}n_{j_{2k}}) dx.
$$
Since $(n_i, p) = 1$, the expression in the brackets is non-zero, as the negative and the positive part are divisible by unequal powers of $p$. Thus, the integral evaluates to zero. 

It follows that only the terms with at least two equal indices survive. Let $j = j_{i_1} = j_{i_2}$. There are three cases: $i_1, i_2 \leq k$, $i_1 \leq k < i_2$ and 
$k < i_1, i_2$. 

For the terms of the first type, one can apply the H\"older inequality and estimate
\begin{eqnarray*}
  \binom {k}{2} \sum_j  \int^{1}_0 F^2_j \sum_{j_2,\ldots,j_{2k}}F_{j_2}F_{j_3}\ldots F_{j_k} \cj{F_{j_{k+1}}}\ldots \cj{F_{j_{2k}}} dx \leq& \\
   \binom {k}{2} \sum_j  \int^{1}_0 |F_j|^2 \left|\sum_i F_i \right|^{2(k-1)} dx  \leq& \\
   \binom {k}{2} \sum_j \| F_j\|^2_{2k} \| \sum_i F_i \|^{2k-2}_{2k}&
\end{eqnarray*}

The remaining two cases can be treated similarly, arriving at the estimate

\begin{eqnarray*}
\| \sum_j F_j\|^{2k}_{2k} \leq \left( k^2 + 2 \binom {k}{2} \right)  \sum_j \| F_j\|^2_{2k} \| \sum_i F_i \|^{2k-2}_{2k} & \\
= \binom {2k}{2} \sum_j \| F_j\|^2_{2k} \| \sum_i F_i \|^{2k-2}_{2k} &.
\end{eqnarray*}
The claim follows.
\end{proof}

\begin{lemma} \label{lm:lambda_complexity_bound}
Let $A \subset \mathbb{Z}$. Then 
   $$
     \lambda_k(A) \leq \binom{2k}{2}^{q(\mc{A})},
   $$
   where $q(\mc{A})$ is the coordinate query complexity of $\mc{A} := \Pi(A)$.
\end{lemma}
\begin{proof}
 Let $\{ w_a \}_{a \in A}$ be arbitrary positive weights with $\sum_{a \in A} w^2_a = 1$. Our goal is to show that 
 $$
 \left\| \sum_{a \in A} w_a  e(ax) \right\|^2_{2k} \leq \binom{2k}{2}^{q(\mc{A})}.
 $$
The strategy is to iteratively apply Proposition~\ref{prop:chang} with a suitable choice of the prime factor $p$ until it is fully reduced to a sum of trivial terms $C_a\|w_a  e(a) \|^2_{2k}$ with some multiplicative factors $C_a$.  If then
$$
C_a \leq \binom{2k}{2}^{q(\mc{A})}
$$ for all $a$, we are done. Indeed, then one can simply write
 $$
 \left\|  \sum_{a \in A} w_a  e(ax) \right\|^2_{2k} \leq \binom{2k}{2}^{q(\mc{A})} \sum_a \left \|w_a  e(ax)  \right\|^2_{2k} = \binom{2k}{2}^{q(\mc{A})} .
 $$

It remains to show how to perform the reduction in such a way. Let $T(\mc{A})$ be the tree defined in Section~\ref{sec:proof_query_PFR} so that its depth is bounded by $q(\mc{A})$. Let $p_0$ be the prime assigned to the root of the tree. By the one-to-one correspondence between the branches of $T(\mc{A})$ and the subsets of $A$ we can decompose $A$ into a disjoint union
$$
A = \bigcup_j p^j_0 A_j
$$
so that the elements of $A_j$ are comprime with $p_0$. Then one can apply Proposition~\ref{prop:chang} with $p = p_0$ and 
$$
F_j(x) := \sum_{a \in A_j} w_{p^j_0a} e(p^j_0 ax).
$$
It follows that
$$
\left\| \sum_{a \in A} w_a  e(ax) \right\|^2_{2k} \leq \binom{2k}{2} \sum_j \left \| F_j \right\|^2_{2k}.
$$
Now we apply the same reduction for each term $\|F_j \|^2_{2k}$ and the subtree of $T(\mc{A}_j)$. However, by construction of $T(\mc{A})$, the depth of each subtree $T(\mc{A}_j)$ is now at most $q(\mc{A}) - 1$. Thus, by induction, it follows that the exponential sum will be fully reduced down to monomials with the multiplicative factors at most $\binom{2k}{2}^{q(\mc{A})}$.

\end{proof}


\subsection{Concluding the proof}

We will need the following standard estimate.

\begin{lemma} \label{lm:lambda_union}
Assume
$$
A = \bigcup_{i} A_i.
$$
Then
$$
\lambda_k(A) \leq \sum_i \lambda_k(A_i).
$$
\end{lemma}
\begin{proof}
Without loss of generality we assume that $A_i$ are disjoint. Let $w_a, a \in A$ be positive weights with $\sum_a w^2_a = 1$. By the triangle inequality
$$
\|\sum_{a \in A} w_a e(ax) \|_{2k} \leq \sum_i \| \sum_{a \in A_i} w_a e(ax) \|_{2k} \leq \sum_i \lambda^{1/2}_k(A_i)\left(\sum_{a \in A_i} w^2_a \right)^{1/2} 
$$
The claim follows by applying Cauchy-Schwarz to the last inequality and squaring both sides.
\end{proof}

Now everything is set up for the proof of Theorem~\ref{thm:few_prods_many_sums_lambda}.

Let $b := \beta_*(A)$ and $\epsilon > 0$ be fixed for the rest of the proof.

\begin{Claim} \label{claim}
Assume $\beta_*(A) \leq b$. Then there is a subset $A' \subset A$ such that $|A'| \geq |A|/2$ and 
$$
\lambda_k(A') \leq 2b^{1/\epsilon}\binom{2k}{2}^{\epsilon \log_2 |A|}.
$$
\end{Claim}
\begin{proof}
Let $T := T(\mc{A})$ be the tree defined in Section~\ref{sec:proof_query_PFR}. By Theorem~\ref{thm:beta} the tree is free of binary subtrees of size $b$. If $|A|/2 \leq b^{1/\epsilon}$ then claim is trivially true for any $A' \subset A$ of size $|A|/2$, so assume the opposite.

By Lemma~\ref{lm:tree_bound}, there's a substree $T' \subset T$ of size at least
$b^{-1/\epsilon}|A|$ and depth at most $\epsilon \log_2 |A|$.  

Let $A'_0$ be the subset of $A$ corresponding to $T'$. We have $q(\mc{A'}) \leq  \epsilon \log_2 |A|$ and so, by Lemma~\ref{lm:lambda_complexity_bound}, 
$$
\lambda_k(A'_0) \leq \binom{2k}{2}^{\epsilon \log_2 |A|}.
$$

Let $A_0 := A \setminus A'_0$. If $|A_0| \leq |A|/2$, stop. If not, repeat the step above and find a set $A'_1 \subset A_0$ of size at least $b^{-1/\epsilon}|A|/2$ and the tree depth at most $\epsilon \log_2 |A|$. We have 
$$
\lambda_k(A'_1) \leq \binom{2k}{2}^{\epsilon \log_2 |A|}.
$$

Reiterating, define $i$ to be the first index such that $|A_i| \leq |A|/2$. Clearly $i \leq b^{1/\epsilon}$ since $|A'_j| \geq b^{-1/\epsilon}|A|/2$ for $0 \leq j < i$.
Put 
$$
A' := \bigcup^i_{j = 0} A'_j.
$$
and estimate 
$$
\lambda_k(A') \leq \sum_j \lambda_k(A'_j) \leq (b^{1/\epsilon}+1)\binom{2k}{2}^{\epsilon \log_2 |A|} \leq 2b^{1/\epsilon}\binom{2k}{2}^{\epsilon \log_2 |A|}.
$$
As $|A'| > |A|/2$ by the choice of $i$, the claim follows.
\end{proof}

Now everything is set up for the proof of Theorem~\ref{thm:few_prods_many_sums_lambda}.
\begin{proof}

Let $A'_0$ be the output of Claim~\ref{claim}. Put $A_0 := A \setminus A'_0$ and apply the claim again to  $A_0$. Note that the hypothesis of Claim~\ref{claim} still holds true as $\beta(A_0) \leq \beta_*(A) = b$. We obtain $A'_1$ such that 
$|A'_1| \geq |A'_0|/2$ and
$$
\lambda_k(A'_1) \leq 2b^{1/\epsilon} \binom{2k}{2}^{\epsilon \log_2 |A_0|} \leq 2b^{1/\epsilon} \binom{2k}{2}^{\epsilon \log_2 |A| - 1}.
$$
Then reiterate with $A_1 := A_0 \setminus A'_1$ to obtain a finite sequence of sets $A'_i$ and $A_i := A_{i-1} \setminus A'_i$ with
$$
\lambda_k(A'_i) \leq 2b^{1/\epsilon} \binom{2k}{2}^{\epsilon \log_2 |A| - i}.
$$
Applying Lemma~\ref{lm:lambda_union} to $\bigcup_i A'_i$ we can crudely estimate
$$
\lambda_k(A) \leq 2 b^{1/\epsilon} \binom{2k}{2}^{\epsilon \log_2 |A|} \sum^{\infty }_{i=0} 2^{-i} \leq 4 b^{1/\epsilon} \binom{2k}{2}^{\epsilon \log_2 |A|} \leq 10 b^\frac{1}{\epsilon} |A|^{2\epsilon \log_2 k}.
$$

\end{proof}

\subsection{Proof of Theorem~\ref{thm:few_prods_many_sums}}

Theorem~\ref{thm:few_prods_many_sums} is a straightforward corollary of Theorem~\ref{thm:few_prods_many_sums_lambda}. By Lemma~\ref{lm:beta-pluneccke}, $\beta_* \leq K_*^2$ and so
$$
|kA| \geq \frac{|A|^k}{\lambda^k_k(A)} \gg_k |A|^{k -  2 \epsilon k \log_2 k}K^{-2k/\epsilon}_*. 
$$




\section*{Acknowledgements}
DZ is supported by a Knut and Alice Wallenberg Fellowship (Program for Mathematics 2017). PD is supported by an MTA Lend\"ulet grant.
The authors thank M\'at\'e Matolcsi, Imre Ruzsa, George Shakan, Misha Rudnev, Oliver Roche-Newton, Ben Green and Ilya Shkredov for useful discussions and feedback on the early drafts. We also thank  the anonymous referees whose comments and suggestions has greatly improved the exposition.

\bibliographystyle{plain}
\bibliography{main}

\end{document}